\documentclass[12pt]{amsart}
\usepackage{graphicx}
\usepackage{subfig}
\usepackage{tabularx}
\usepackage{tikz}
\usepackage{tikzscale}
\usepackage{pdflscape}
\usepackage{graphicx}
\begin{document}
\newtheorem{proposition}{Proposition}
\newtheorem{corollary}{Corollary}
\newcommand{\ds}{\displaystyle}
\newtheorem{definition}{Definition}[section]
\newtheorem{theorem}{Theorem}
\newtheorem{lemma}{Lemma}
\newcommand\scaleSizeA{0.22}
\newcommand\scaleSizeB{0.14}
\parskip=12pt


\title[Edge conflicts and geodesics]{Edge conflicts do not determine geodesics in the associahedron}

\author{Sean Cleary}
\address{Department of Mathematics \\
The City College of New York and the CUNY Graduate Center\\
City University of New York \\
New York, NY 10031}
\email{cleary@ccny.cuny.edu}
\urladdr{{\tt http://www.sci.ccny.cuny.edu/$\sim$cleary}}

\author{Roland Maio}
\address{Department of Computer Science \\
The City College of New York\\
City University of New York \\
New York, NY 10031}
\email{rolandmaio38@gmail.com}

\thanks{
	This material is based upon work supported by the National Science Foundation under Grant No. \#1417820.
}
\keywords{random binary trees}

\begin{abstract}
There are no known efficient algorithms to calculate distance in the one-skeleta of associahedra, a problem 
that is equivalent to finding rotation distance between rooted binary trees or the flip distance between polygonal triangulations.
One measure of the difference between trees is the number of conflicting edge pairs, and a natural way of trying to find short paths is to minimize successively this number of conflicting edge pairs using flip operations in the corresponding triangulations.
We describe examples that show that the number of such conflicts does not always decrease along geodesics.  Thus, a greedy algorithm that always chooses a transformation that reduces conflicts will not produce a geodesic in all cases.  Further, for any specified amount, there are examples of pairs of all large sizes showing that the number of conflicts can increase by that amount along any geodesic between  the pairs.
\end{abstract}

\maketitle

\section{Introduction}

Rooted binary trees arise across a range of areas, from phylogenetic trees representing genetic relationships to efficient organizational structures in large datasets.   When considering two rooted binary trees, there are a wide
range of possible measures of distance between them, depending upon the situation and how much structure we attach to the trees.   In the setting of rooted binary trees with a natural right-to-left order, such as those corresponding to binary search trees, a widely-considered distance is rotation distance.  A {\em left rotation} at a node $n$ of a tree $S$ is an operation which promotes the rightmost grandchild of $n$ to become a right child of $n$, changes the left child of the right child of $n$ to become the right child of the left child of $n$, and demotes the left child of $n$ to become the leftmost grandchild of $n$, as pictured in Figure \ref{fig:exampletree1}, where the subtree $e$ containing the sibling pair $(4,5)$ of $d$ is promoted from rightmost grandchild of $b$ to become the right child of $b'$, and leaf 3 remains a grandchild but is moved to the left, and node $c$ containing the sibling pair $(1,2)$ is demoted from left child to leftmost grandchild.  A {\em right rotation} is the inverse operation, for example taking the right tree of Figure  \ref{fig:exampletree1} to the left one. The {\em rotation distance $d(S,T)$} between two trees $S$ and $T$ is the minimum number of rotations needed to transform the tree $S$ to the tree $T$.  There is a natural bijection between rooted binary trees with $n$ internal nodes (or $n+1$ leaves) and triangulations of a marked regular polygon with $n+2$ sides, shown in Figure \ref{fig:exampletree1}, where the tree is the dual to the triangulation and the root of the tree corresponds to the edge marked $R$ in the triangulation. A rotation in a binary tree corresponds to an edge flip in a triangulation as pictured. Thus any sequence of rotations between two trees is exactly equivalent to a corresponding sequence of edge flips between triangulations, as will be described below.

The associahedron of size $n$ is a combinatorial object capturing many aspects of Catalan objects, including triangulations, trees, and bracketed expressions.
 Here, by the associahedron graph of size $n$ we mean the graph whose vertices are triangulations of a marked regular $(n+2)$-gon (equivalently, rooted trees with $n+1$ leaves) and where two vertices $S$ and $T$ are connected if there is a single edge flip that transforms the triangulation $S$ to $T$ (equivalently, if the associated rooted trees differ by a single rotation at a node.)  This graph is the one-dimensional skeleton of  higher-dimensional polytopes known as associahedra.
 Here, we work entirely in the one-dimensional skeleton and neglect the higher dimensional faces in the face lattices of associahedra.
 
 \begin{figure}
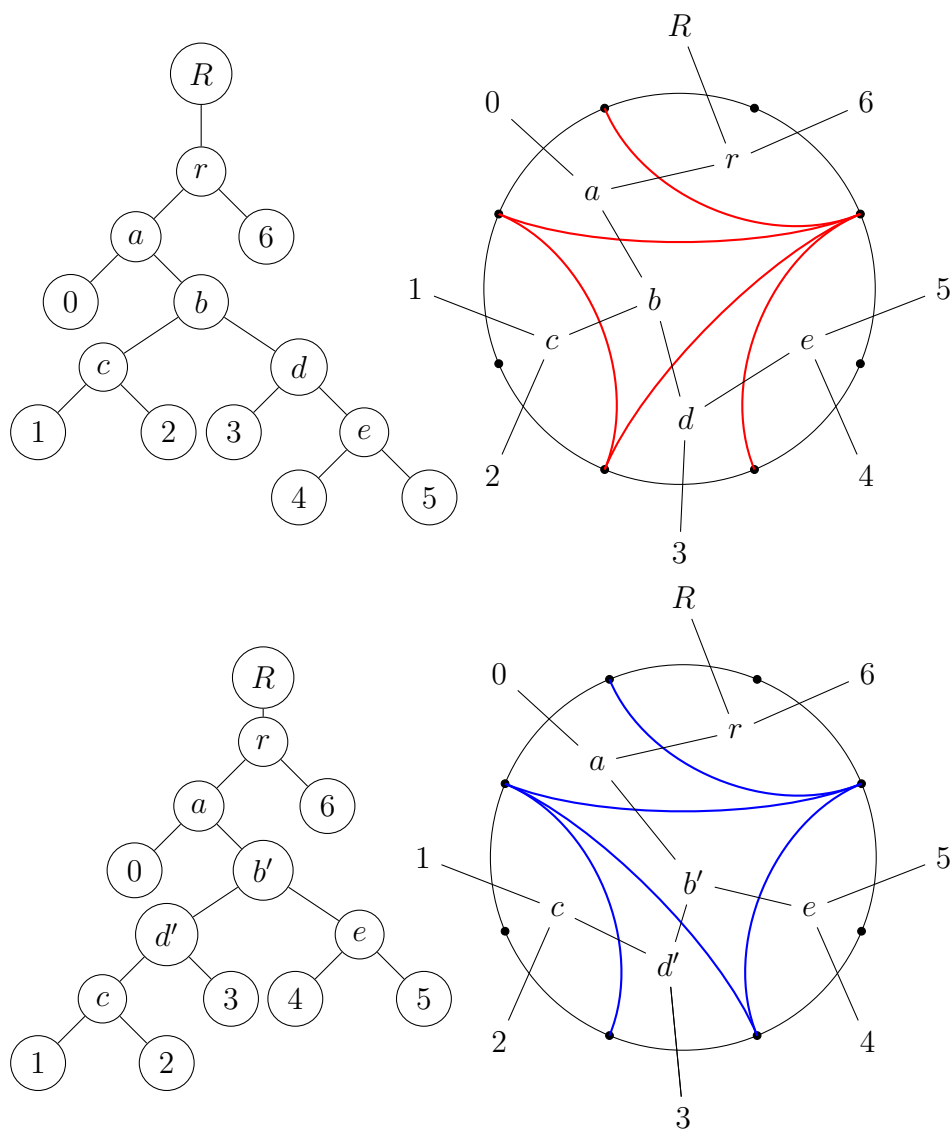

 \includegraphics[width=\textwidth]{tree1.tikz}
  \includegraphics[width=\textwidth]{tree2.tikz}
 
    \caption{The rooted binary tree with encoding 1101100101000 on the top left and on the top right its corresponding triangulation. This tree differs by a single rotation (left at node $b$) from the 
     tree 1101110001000 on the lower left and with triangulation on the lower right. The triangulation differs by a single edge flip, flipping the edge between regions $b$ and $d$ from top to bottom.
    A right rotation at the node labelled $b'$ in the lower left tree transforms back to the top left tree, and flipping the edge between $b'$ and $d'$ will transforms to the upper right triangulation.
    }

    \label{fig:exampletree1}

\end{figure}
  
 There are no known polynomial-time algorithms to find shortest paths in associahedra graphs or to find distances between vertices in those graphs.  There are a number of approximation algorithms \cite{barilpallo, clearylinear} for rotation distance.  Two edges are said to be {\em conflicting } if it is not possible for them to be part of the same triangulation, due to the two edges crossing.   This is the ordered version of having two edges be in conflict in trees without an order on leaves, such as those arising in phylogeny.  Edges that are not in conflict are said to be {\em compatible}, and two trees that have a large number of compatible edges are considered reasonably close.  See Semple and Steel \cite{charlesmike} for discussion of combinatorial and algorithmic questions about counting the number of compatible edges (or the number of edge conflicts)  between phylogenetic trees, including estimates of distance from increasing the number of compatible edges, which is equivalent to reducing the number of conflicting edges.
 
 Here, we investigate connections between edge conflicts and finding geodesics (minimal length paths) in the setting of ordered trees, or equivalently of triangulations of polygons.  Counting conflicts gives a crude estimate of distance in that 
 two triangulations that have many  conflicting edge pairs are generally not close together, and triangulations that have few conflicts are typically closer together.  The number of expected conflicts between triangulations selected uniformly at random has been analyzed experimentally by Chu and Cleary \cite{conflicts} and appears to grow slightly more than linearly with size.    Here, we describe example triangulation pairs  of size 8 and larger where there is no geodesic along which the total number of conflicts between the triangulations uniformly decreases or even remains the same.  This rules out the potential success of various greedy algorithms to reduce conflicts to find shortest paths or distances in these associahedra graphs.  In Section \ref{secbig}, we show that the required increase of conflicts along a geodesic can grow to any specified increase $k$ for triangulation pairs of size $k+7$ and larger.
 
 The authors are grateful to the anonymous reviewers for numerous helpful suggestions for clarity and correctness, as well as suggesting an approach for more generality.

\begin{figure}
 \includegraphics[width=\textwidth]{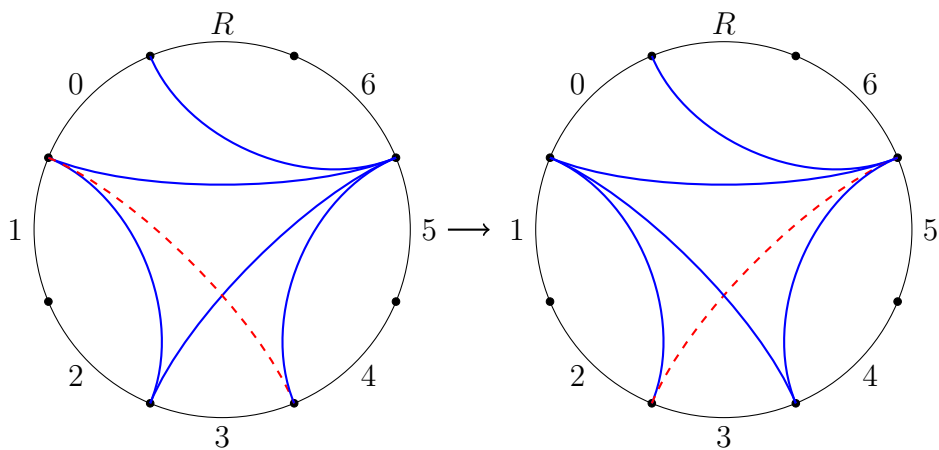}

    \caption{The edge flip which transforms the left triangulation of Figure \ref{fig:exampletree1} to that of the right one of Figure \ref{fig:exampletree1} in blue, with the red dashes edges indicating the edges involved in the change.} 
    \label{fig:flip}

\end{figure}

\section{Background and definitions}

We consider the marked regular $(n+2)$-gon $P$ with edges labelled as $R$ for a marked root edge and then consecutively numbered counter-clockwise from $0$ to $n$. A {\em triangulation $T$} of $P$ is a collection of $n-1$ non-crossing edges from vertices of $P$ that separate $P$ into $n$ triangles.  An {\em edge flip on $T$} is the process of taking two triangles that share an interior edge, thus forming a quadrilateral $Q$, and replacing the interior diagonal of $Q$ that lies in $T$ with the other diagonal of $Q$ to form a new triangulation $T'$.  For each $n$, the relevant {\em associahedron graph} is  the graph whose vertices are triangulations of the regular $(n+2)$-gon and whose edges connect vertices that differ by a single edge flip.  The {\em edge flip distance} from a triangulation $S$ to a triangulation $T$ of the same size is the minimal path length in the associahedron graph.  Alternatively, we can regard the vertices  of the size $n$ associahedron graph as rooted binary trees of size $n$ interior nodes with $n+1$  leaves numbered from $0$ to $n$, with the edges connecting vertices whose trees differ by a single rotation (left or right) at an interior node.  To any triangulation, there is a well-known natural dual construction giving rise to a description involving rooted trees illustrated in Figures \ref{fig:exampletree1} and \ref{fig:flip}. For further background see, for example, Stanley \cite{stanley1}.  Edge flip distance between triangulations corresponds exactly to {\em rotation distance} between the dual binary trees.

By {\em binary tree} in the following, we mean a rooted tree where every node has either 0 children (and is thus a {\em leaf}) or 2 children (and is thus an {\em interior node}) with an order on the children, resulting in a left-to-right order on the nodes and leaves.
Rotation distance was described by Culik and Wood \cite{cw} whose arguments showed that the one-skeleta of associahedra are connected and gave an upper bound of $2n-2$ on the distance between any two vertices of the associahedron  graph of size $n$.  Remarkable work of Sleator, Tarjan and Thurston \cite{stt} showed that a better upper bound for distance between vertices for $n \geq 11$ is $2n-6$, and furthermore that upper bound is realized for all $n$ larger than some very large $N$.  Recent work of Pournin \cite{pournin} showed that in fact the upper bound is realized for all $n \geq 11$.   There are no known polynomial-type algorithms to compute rotation distance, although rotation distance has been shown to be fixed parameter tractable by Cleary and St.~John \cite{rotfpt} and there are a number of approximation algorithms \cite{barilpallo, clearylinear}.  A {\em geodesic} between two trees is a minimal length path between the trees, whose length is necessarily the rotation distance between the endpoints.

We describe trees via the binary sequence obtained by a pre-order traversal of all nodes of the tree, recording a 1 for each internal node and 0 for each leaf node, giving a binary sequence of $n$ 1's and $n+1$ 0's via this encoding.  So for example the balanced tree with four leaves is encoded as 1100100.   In the figures, we draw triangulations as collections of chords in the hyperbolic disk with the boundary as an $(n+2)$-gon, with the intervals numbered counterclockwise from $0$ to $n$ and the root interval either labelled $R$ or unlabelled, at the top.  The counterclockwise end of the interval labeled $i$ is referred to as vertex $i$ when needed.

Two edges $s$ and $t$ are said to be {\em conflicting} if it is not possible for them to be present in the same triangulation; equivalently, if the edges cross as chords connecting vertices in the relevant polygon. 
Two edges that are not in conflict are said to be {\em compatible}, a term used widely in phylogenetic settings, see Semple and Steel \cite{charlesmike}. For example, a chord from vertex 5 to 8 is in conflict with a chord from vertex 3 to 6, which can be seen as one of the intersections of the red chords and blue chords in Figure \ref{sandt} if we number the vertices as the counterclockwise ends of the numbered edges. The number of conflicts for a pair of triangulations $S$ and $T$ is the sum of the conflicts between the pairs of edges.  We note that the total number of conflicts between $S$ and $T$ is an estimate of the distance between $S$ and $T$ only weakly.  In particular, the number of conflicts between $S$ and $T$ is a semi-metric in that it is symmetric and positive definite and satisfies that if the total number of conflicts is zero, the triangulations must coincide.  However, it is not a metric in that it does not satisfy the triangle inequality, as can be seen by considering two triangulations  $A$ and $B$ of the pentagon that are not adjacent in the associahedron graph.  They have a common neighbor $C$ with which $A$ and $B$ each have one conflict, but the number of conflicts from $A$ to $B$ is three.

\section{Conflicts along geodesics \label{secexample}}

There are a range of conflict-based greedy algorithms to reduce conflicts to try to find a geodesic from a given triangulation $S$ to a given target triangulation $T$.  Generally, they proceed in a manner under the following approach:

\begin{enumerate}
\item Set $S=S_0$ to begin
\item If $S_i$ is $T$, then we are done and the distance is no more than $i$.
\item If $S_i$ is not the triangulation $T$, then we calculate the conflicts between all of the neighbors of $S_i$ and the target $T$.  Among those neighbors, we let $S_{i+1}$ be a triangulation with minimal conflicts, and then we repeat.
\end{enumerate}

The simplest version is merely to always take the neighbor with minimal conflicts that is lexicographically least.   There are other more sophisticated approaches for choosing among ties between neighbors using some type of greedy approach but we will see below that any algorithm of this type is guaranteed to give an overestimate in some cases.

We note that this algorithm will produce a path from $S$ to $T$ since unless the triangulations coincide, there is always a neighbor with fewer conflicts with the target triangulation (see Hanke, Ottomann, and Schuierer \cite{hankeottomannschuierer} for a proof of this in the more general case of triangulations of point sets.)  Thus the algorithm will always give a path from $S$ to $T$, and this results in an upper bound for the edge-flip distance.  However, this path is not necessarily a geodesic path so the distance from $S$ to $T$ may be less than various greedy conflict-based algorithms may find.
 There are multiple ways in which such a greedy algorithm to reduce conflicts can fail to find a geodesic.   It may be that there are several choices among the neighbors of $S_i$ with the same number of conflicts, and at least one of them does not begin a geodesic path from $S_i$ to $T$. Or it may be that none of the neighbors with minimal conflicts begins a geodesic path.  Broadening the notion to allow choices where the number of conflicts is not necessarily minimal but merely equal or smaller than the current number of conflicts gives many more possibilities.  But even algorithms that consider reducing conflicts (not necessarily minimally) or keeping conflicts constant will not always find a geodesic path because of examples of the following type:
 \begin{figure} 
\centering
\begin{tabular}{cc}
	\subfloat[$S$]{\label{fig:1}\includegraphics[scale=\scaleSizeA]{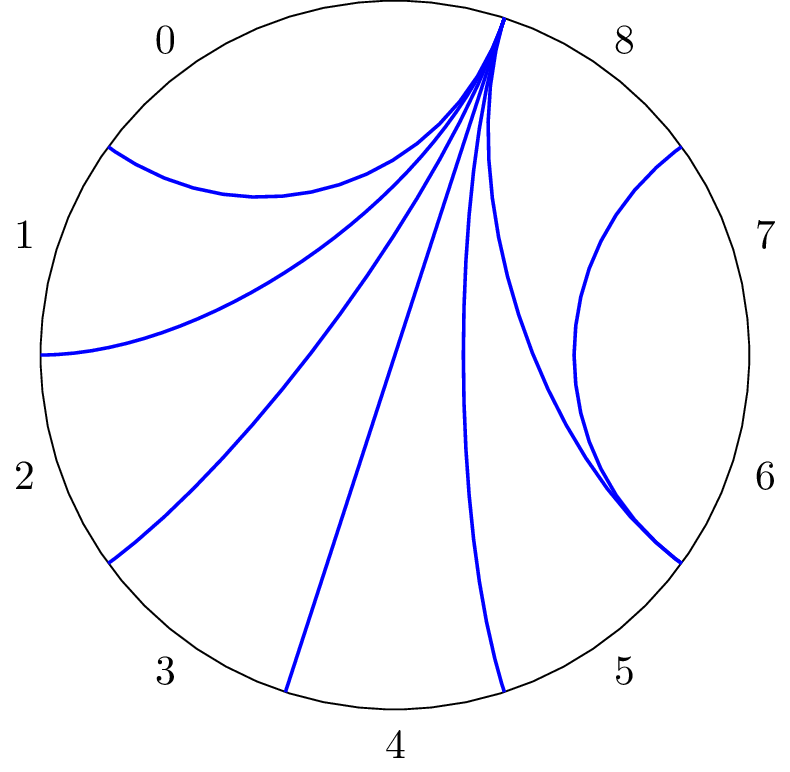}} &
	\subfloat[$T$]{\label{fig:2}\includegraphics[scale=\scaleSizeA]{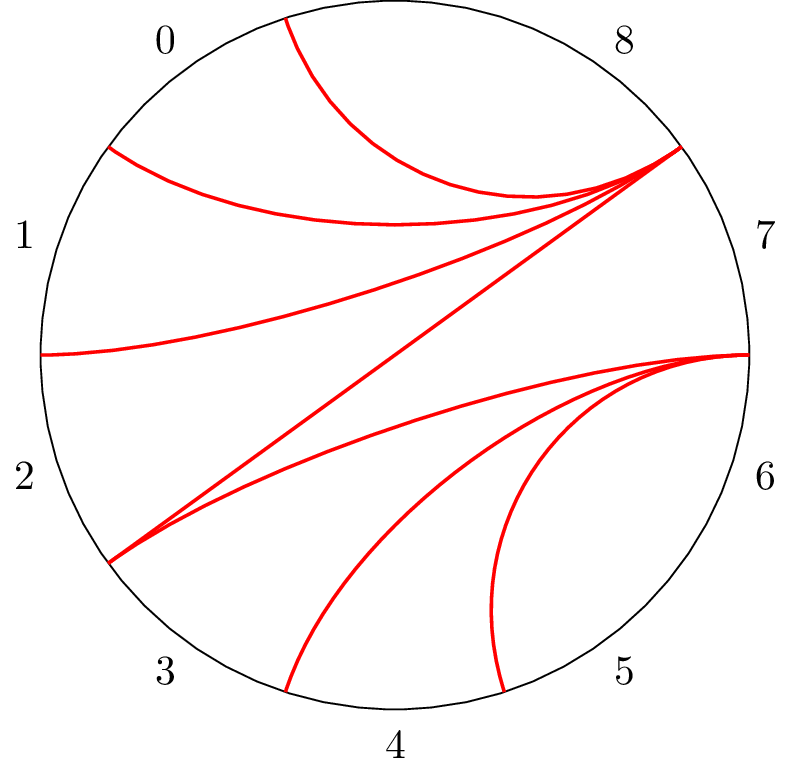}}
\end{tabular}
\caption{ \label{sandtsep} The pair of triangulations $S$ and $T$ used to prove Theorem 1.}

\end{figure}
 
\begin{theorem}
For any $n$ at least 8, there are examples of triangulation pairs $(S,T)$ of size $n$ where every geodesic $\gamma$ from $S$ to $T$ with $\gamma = \{S=S_0, S_1, S_2, \ldots, S_k=T \}$  has the property that $S_1$ has more conflicts with $T$ than $S$ has with $T$.
\end{theorem}

\begin{proof} We consider a pair of triangulations $S$ and $T$, each with 7 chords dividing the  10-gon into 8 triangles.  Their dual trees correspond to trees with 8 internal nodes and 9 leaves.  The triangulation $S$ has encoding $10101010101011000$ for its dual tree and the dual tree for  $T$ has encoding $11010101101010000$. The triangulations $S$ and $T$ are shown in Figure \ref{sandtsep}.  The distance in the associahedron graph is 8, obtained by a breadth-first enumeration of neighborhoods of increasing size of $S$ until reaching $T$.  The number of conflicts between $S$ and $T$ is 27 total conflicts, as shown in Figure \ref{sandt} where the triangulations are superimposed.   There are 7 neighbors to $S$, each corresponding to flipping one of the seven edges.  The number of conflicts of those neighbors of S with T are 27, 26, 26, 23, 22, 22 and 28.  The only neighbor of $S$ that is closer to $T$ than $S$ is the last one, that has more conflicts than the original $S$.  Thus any geodesic from $S$ to $T$ will necessarily have the conflicts rise from 27 to 28.   Note that we can add identical additional triangles to each of the triangulations in the pair to create examples of this type for any size at least 8, since these additional triangles will not be disturbed along any geodesic by Lemma 3b of Sleator, Tarjan and Thurston \cite{stt}, giving essentially the same geodesic and conflicts along the geodesic.  Note further that though this proof is computer-assisted in analyzing geodesics by exhaustion, this can also be proved directly using the methods described in the proof of Theorem \ref{thmbiggap} \end{proof}

Thus, any algorithm that attempts to find geodesics by either minimizing conflicts, at least reducing conflicts, or keeping conflicts at least non-increasing will not find any geodesic from $S$ to $T$.

We note a few properties of this particular example.  Though it is quite common to have a multitude of geodesics between a pair of triangulations, computation shows that in this case there is a unique geodesic $\gamma$ from $S$ to $T$.  This unique geodesic proceeds through triangulations listed in Table \ref{gammatab} which are pictured in Figure \ref{gamma}.

\begin{figure}
\includegraphics[scale=\scaleSizeA]{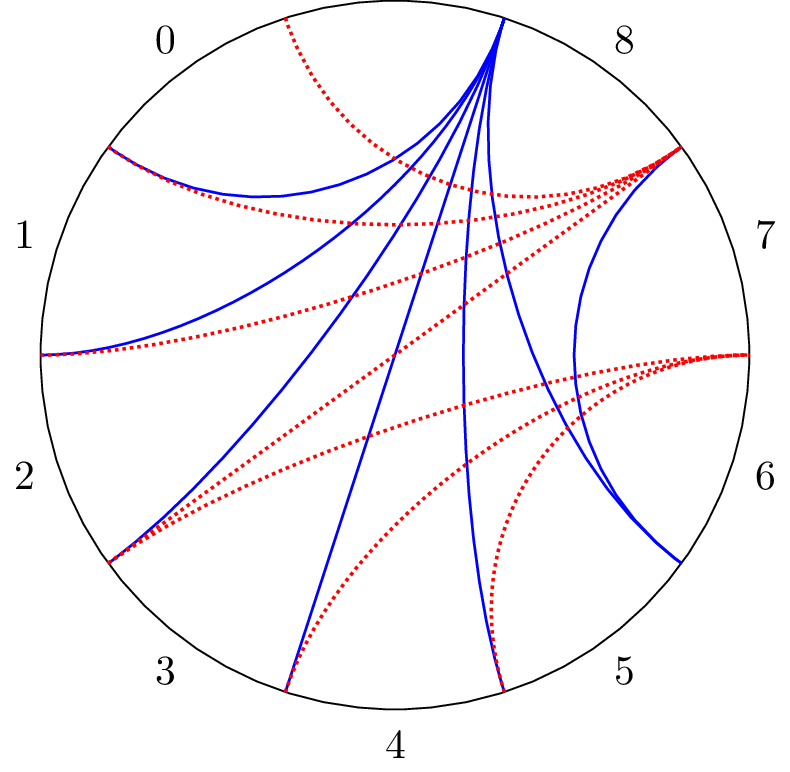}
\caption{\label{sandt} The triangulations $S$ and $T$ superimposed, with $S$ in blue and $T$ in dotted red.  There are 27 conflicts in the pair $(S,T)$ which can be seen as the 27 intersections between the triangulations.}

\end{figure}

\begin{table}
\begin{tabular}{|c|c|} \hline
Triangulation & Conflicts with $T$  \\ \hline

$S=S_0$ & 27 \\ \hline
$S_1= 10101010101010100$ & 28 \\ \hline
$S_2= 10101010101100100$ & 21 \\ \hline
$S_3= 10101010110100100$ & 15 \\ \hline
$S_4= 10101011010100100$ & 10 \\ \hline
$S_5= 10101011101010000$ & 6 \\ \hline
$S_6= 10101101101010000$ & 3 \\ \hline
$S_7= 10110101101010000$ & 1 \\ \hline
$T$ & 0 \\ \hline
\end{tabular}  
\caption{The triangulations in the unique geodesic $\gamma$ from $S$ to $T$ and the conflicts with the target triangulation $T$ along this path of length 8}
\label{gammatab}
\end{table}

\begin{figure}
\centering
\begin{tabular}{ccc}
	\subfloat[$S$]{\includegraphics[scale=\scaleSizeB]{SandTnhl.png}} &
	\subfloat[$S_1$]{\includegraphics[scale=\scaleSizeB]{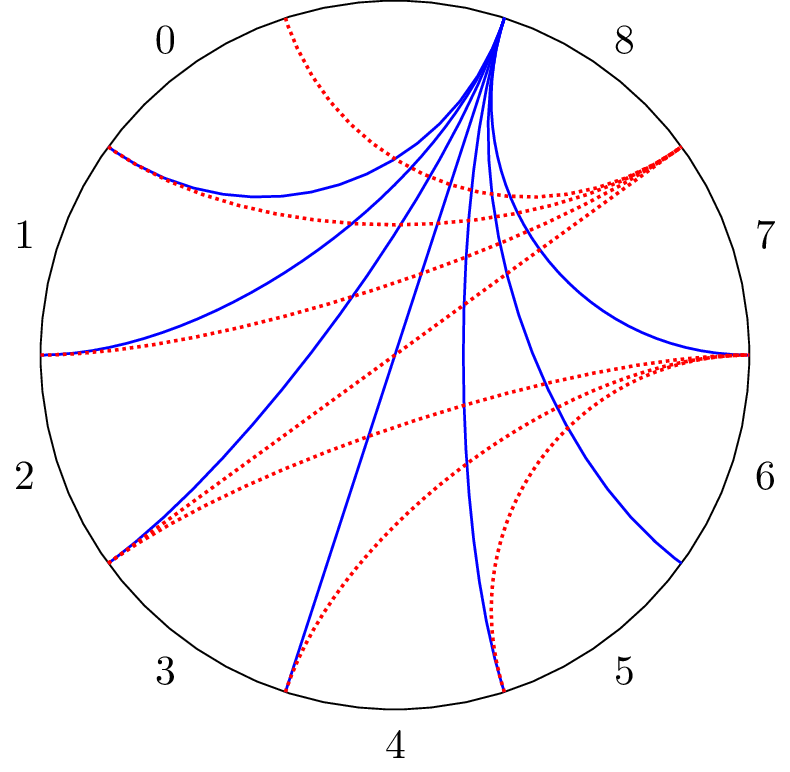}} &
	\subfloat[$S_2$]{\includegraphics[scale=\scaleSizeB]{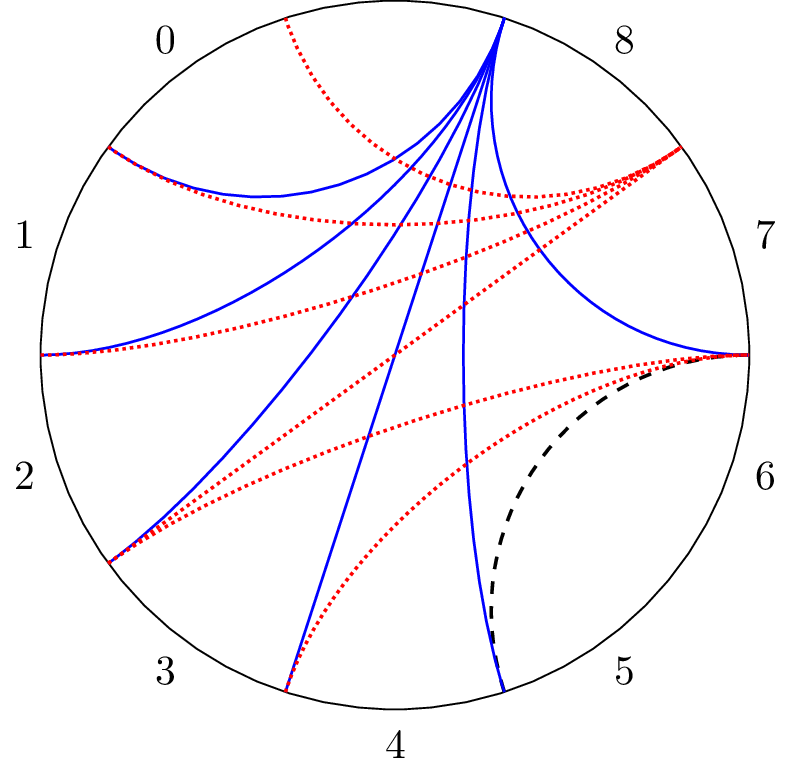}} \\
	\subfloat[$S_3$]{\includegraphics[scale=\scaleSizeB]{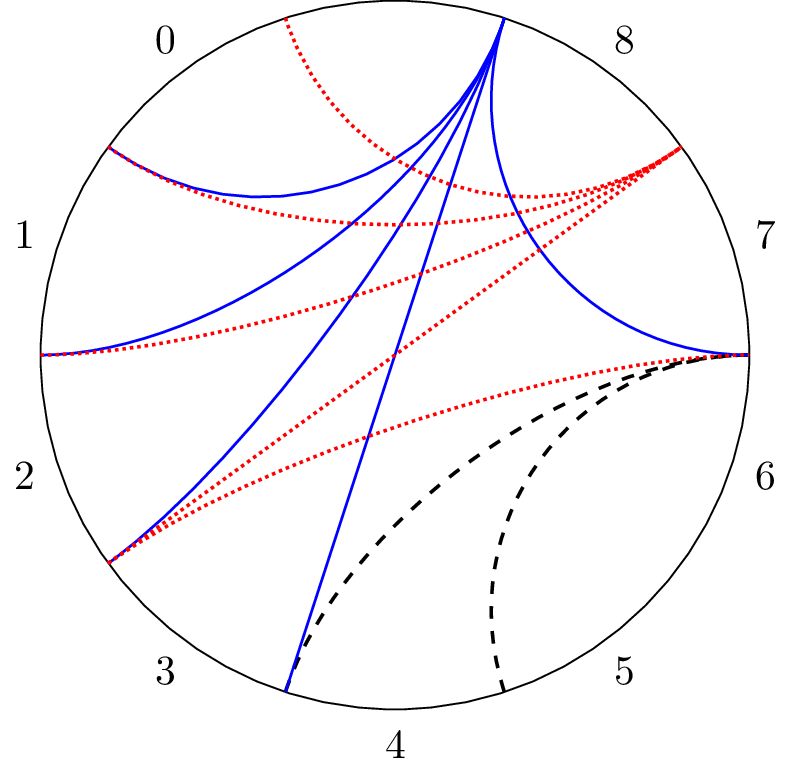}} &
	\subfloat[$S_4$]{\includegraphics[scale=\scaleSizeB]{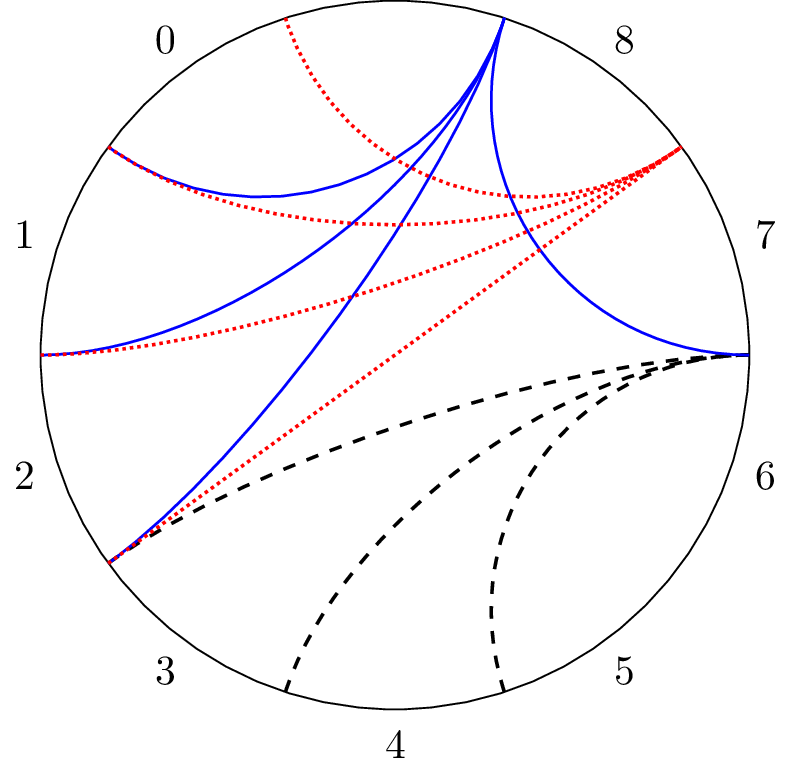}} &
	\subfloat[$S_5$]{\includegraphics[scale=\scaleSizeB]{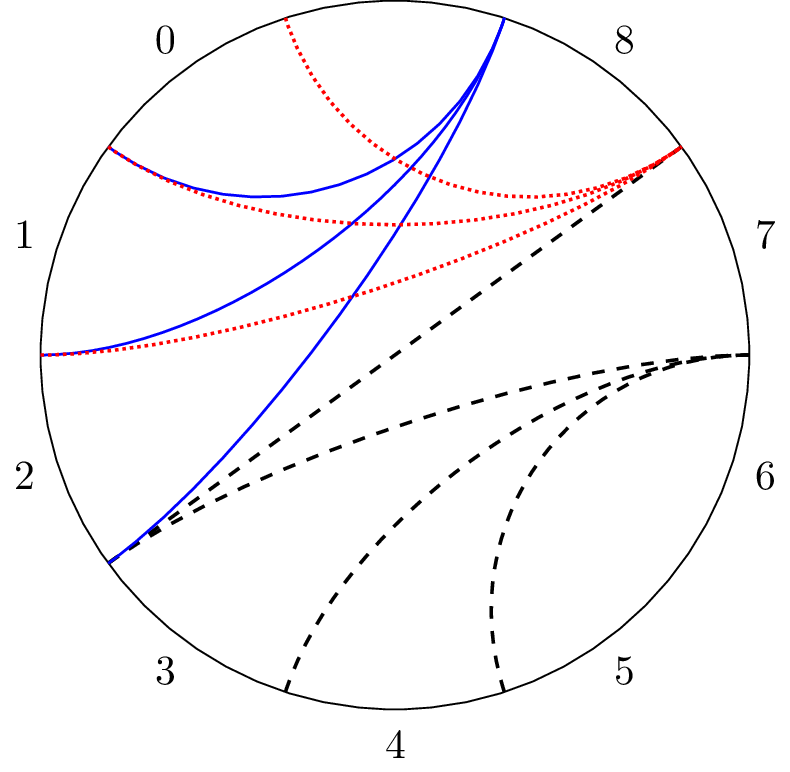}} \\
	\subfloat[$S_6$]{\includegraphics[scale=\scaleSizeB]{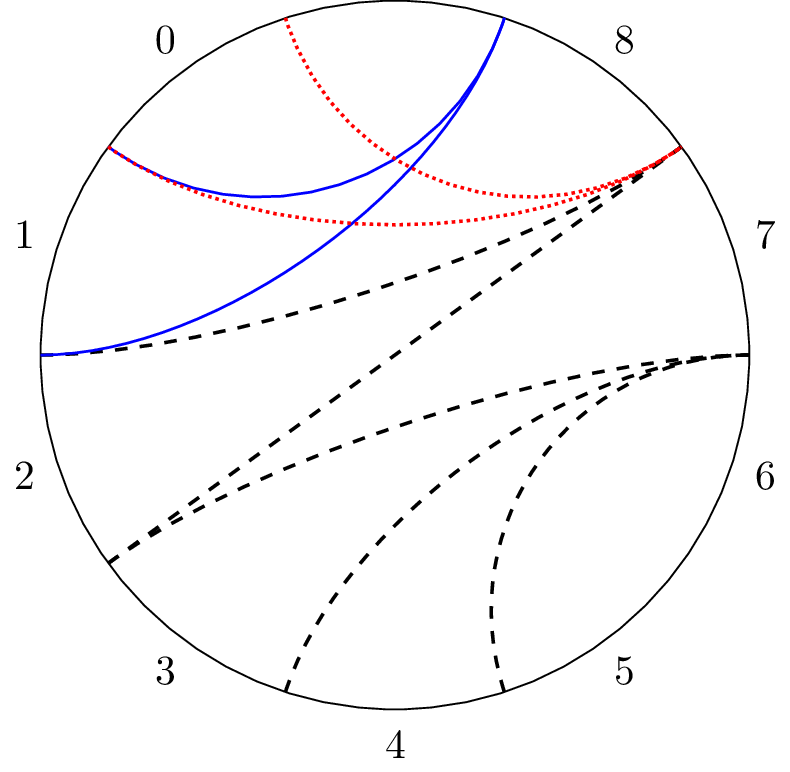}} &
	\subfloat[$S_7$]{\includegraphics[scale=\scaleSizeB]{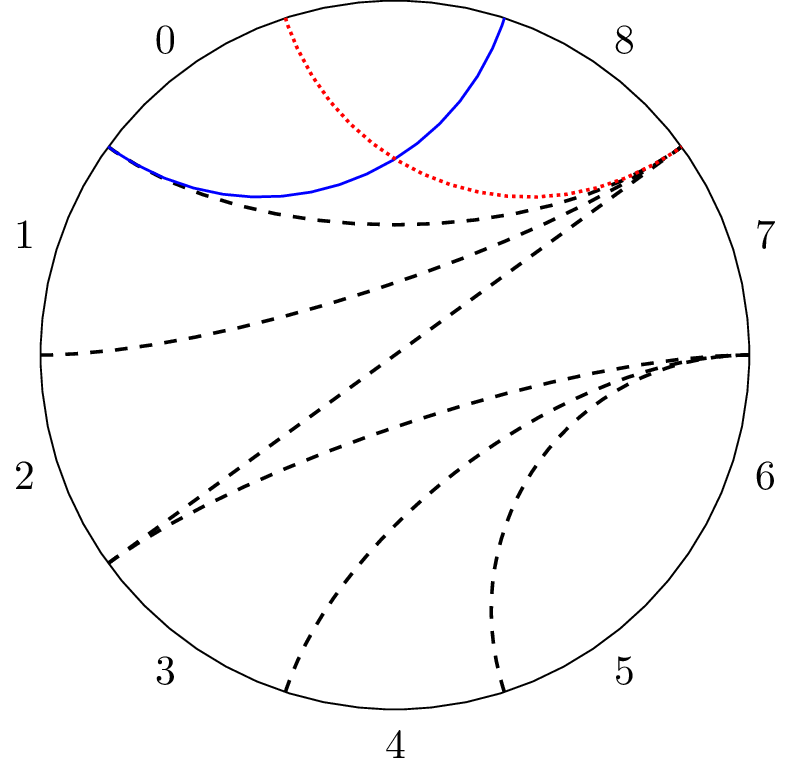}} &
	\subfloat[$T$]{\includegraphics[scale=\scaleSizeB]{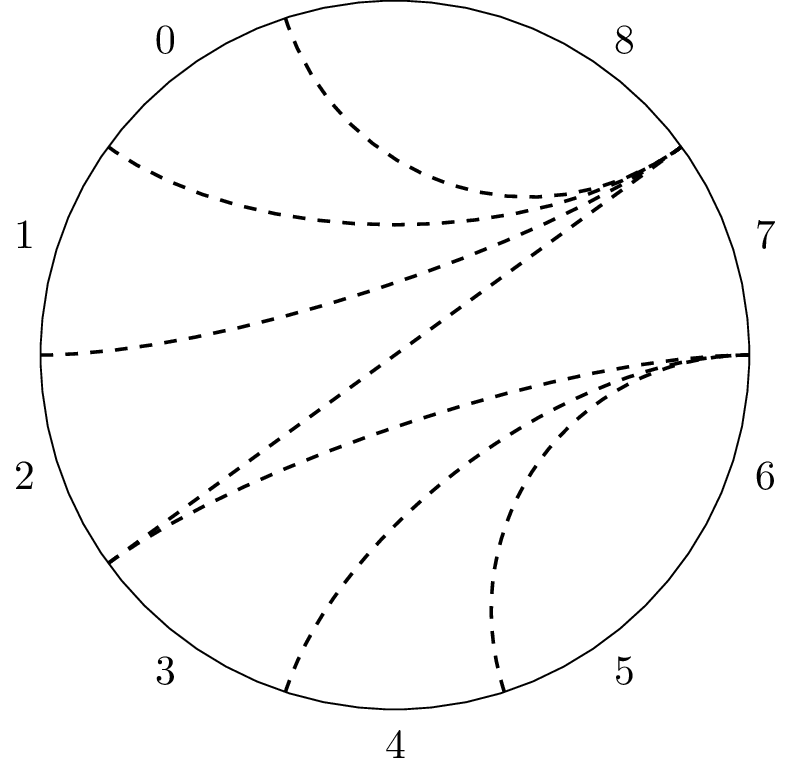}}
	 
\end{tabular}
\caption{The triangulations in the geodesic $\gamma$ from $S$ to $T$ superimposed on the target $T$.  The triangulation $S$ is in solid blue and $T$ in dotted red, and common edges are shown as dashed black.}
\label{gamma}
\end{figure}

 This geodesic $\gamma$ begins with a first move that actually has  the maximum number of conflicts among all neighbors of $S$.  The greedy algorithm, taking the lexicographically first minimal conflict neighbor in each case, gives a path of length 9, whereas the minimal length path is of length 8, thus giving an over-estimate by 1.

The example pairs $S$ and $T$ given above are one of 40 equivalence classes (up to rotational and reflectional symmetry of the triangulations) of examples of size 8 where the initial geodesic step must increase the number of conflicts, with representatives of all equivalence classes listed in Table \ref{bad8}. 

This particular example is not symmetric.  We note that if we proceed instead from $T$ toward $S$, the neighbors of $T$ have conflicts 21, 23, 24, 29, 25, 25, and 26 with $S$, with the only neighbor of $T$ that is closer to $S$ having 21 conflicts, the smallest of all choices with respect to conflicts.   And, in fact, the greedy algorithm that proceeds from the $T$ toward $S$ in steps by always taking the lexicographically least neighbor with the fewest conflicts with $S$ does find a geodesic path of length 8.   However, we can construct triangulation pair examples that have the property of conflicts rising along all geodesics from both ends.

\begin{theorem} \label{biconflict}
There are examples of triangulation pairs $(U,V)$ where for all geodesic paths  $\gamma$ from $U$ to $V$ or from $V$ to $U$ there is a step where the number of conflicts increases along $\gamma$.
\end{theorem}

\begin{proof}

We concatenate the triangulation pair above  $(S,T)$ with its opposite pair $(T,S)$ along any of the peripheral edges to obtain a pair of triangulations $(U,V)$ of the 18-gon.  The peripheral edge used in the concatenation is now a common edge and does not flip during any geodesic path by Lemma 3b of Sleator, Tarjan and Thurston \cite{stt}.   Thus by the arguments above, for any geodesic $\gamma$ from $U$ to $V$ or from $V$ to $U$, there must be at least one point along $\gamma$ where the total number of conflicts must increase.
\end{proof}

We note that there are smaller example pairs than the ones obtained by doubling used in the proof of Theorem \ref{biconflict}.  For example, the triangulations in the pair of Figure \ref{biup9} are of size 9 and have the property that every geodesic from $S$ to $T$ or from $T$ to $S$ begins with an increase in conflicts.

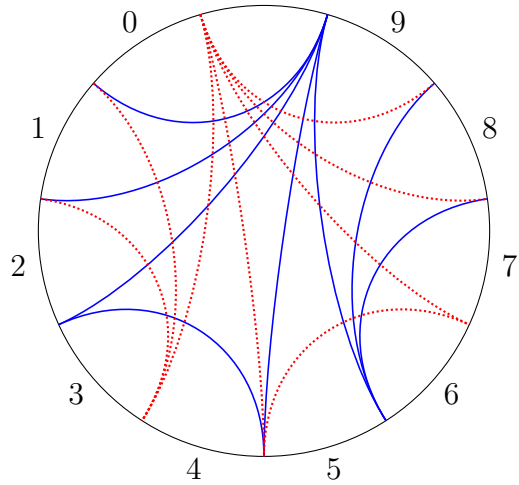
\begin{figure}
\begin{tikzpicture}[scale=3]

    \tikzstyle{single conflict node}=
        [draw=yellow!30!black, fill=yellow, opacity=0.8,
         shape=circle, scale=0.4]
    \tikzstyle{multiple conflict node}=
        [draw=orange!40!black, fill=orange, shape=circle, scale=0.6]
	\tikzset{chord of S/.style={blue, semithick}}
	\tikzset{chord of T/.style={red, thick, densely dotted}}
	\tikzset{common chord/.style={black, thick, dashed}}

    \draw (0,0) circle [radius=1];
	\node(0) at (-0.5947,0.9254) {$0$};
	\node(1) at (-1.0006,0.4570) {$1$};
	\node(2) at (-1.0888,-0.1565) {$2$};
	\node(3) at (-0.8313,-0.7203) {$3$};
	\node(4) at (-0.3099,-1.0554) {$4$};
	\node(5) at (0.3099,-1.0554) {$5$};
	\node(6) at (0.8313,-0.7203) {$6$};
	\node(7) at (1.0888,-0.1565) {$7$};
	\node(8) at (1.0006,0.4570) {$8$};
	\node(9) at (0.5947,0.9254) {$9$};

	\coordinate(v0) at (-0.2817,0.9595) {};
	\coordinate(v1) at (-0.7557,0.6549) {};
	\coordinate(v2) at (-0.9898,0.1423) {};
	\coordinate(v3) at (-0.9096,-0.4154) {};
	\coordinate(v4) at (-0.5406,-0.8413) {};
	\coordinate(v5) at (-0.0000,-1.0000) {};
	\coordinate(v6) at (0.5406,-0.8413) {};
	\coordinate(v7) at (0.9096,-0.4154) {};
	\coordinate(v8) at (0.9898,0.1423) {};
	\coordinate(v9) at (0.7557,0.6549) {};
	\coordinate(v10) at (0.2817,0.9595) {};
	\draw[chord of S]
		(v1) .. controls (-0.3779,0.3274) and (0.1409,0.4797) .. (v10);
	\draw[chord of S]
		(v2) .. controls (-0.4949,0.0712) and (0.1409,0.4797) .. (v10);
	\draw[chord of S]
		(v3) .. controls (-0.4548,-0.2077) and (0.1409,0.4797) .. (v10);
	\draw[chord of S]
		(v3) .. controls (-0.4548,-0.2077) and (-0.0000,-0.5000) .. (v5);
	\draw[chord of S]
		(v5) .. controls (-0.0000,-0.5000) and (0.1409,0.4797) .. (v10);
	\draw[chord of S]
		(v6) .. controls (0.2703,-0.4206) and (0.1409,0.4797) .. (v10);
	\draw[chord of S]
		(v6) .. controls (0.2703,-0.4206) and (0.3779,0.3274) .. (v9);
	\draw[chord of S]
		(v6) .. controls (0.2703,-0.4206) and (0.4949,0.0712) .. (v8);
	\draw[chord of T]
		(v0) .. controls (-0.1409,0.4797) and (0.3779,0.3274) .. (v9);
	\draw[chord of T]
		(v0) .. controls (-0.1409,0.4797) and (0.4949,0.0712) .. (v8);
	\draw[chord of T]
		(v0) .. controls (-0.1409,0.4797) and (0.4548,-0.2077) .. (v7);
	\draw[chord of T]
		(v0) .. controls (-0.1409,0.4797) and (-0.0000,-0.5000) .. (v5);
	\draw[chord of T]
		(v0) .. controls (-0.1409,0.4797) and (-0.2703,-0.4206) .. (v4);
	\draw[chord of T]
		(v5) .. controls (-0.0000,-0.5000) and (0.4548,-0.2077) .. (v7);
	\draw[chord of T]
		(v1) .. controls (-0.3779,0.3274) and (-0.2703,-0.4206) .. (v4);
	\draw[chord of T]
		(v2) .. controls (-0.4949,0.0712) and (-0.2703,-0.4206) .. (v4);

\end{tikzpicture}
\caption{Two triangulations superimposed, with $S$ in blue having associated tree with encoding $1010101100101110000$ and $T$ in dotted red having tree with encoding  $1111110101000100000$.  There are 33 conflicts between $S$ and $T$ and the unique first geodesic step from $S$ to $T$ flips the edge from 2 to 4 to an edge from  3 to 9 giving 34 conflicts between the resulting triangulations. Similarly, the unique first geodesic step from $T$ to $S$ flips the edge from 4 to 6 to an edge from  5 to the root vertex giving 34 conflicts between the resulting triangulations. \label{biup9}}
\end{figure}

\section{Greater rises in conflicts \label{secbig}}

The particular example triangulation pair of Figure \ref{sandtsep} from Section \ref{secexample} shows that conflicts can necessarily rise along geodesics by one.  In fact, for increasingly large triangulation pairs, conflicts along geodesics can rise by any specified amount.  We use  methods similar to  analyses in Cleary, Rechnitzer, and Wong \cite{commonedges} and Cleary and St.~John \cite{clearylinear} to bound distances sharply.   A {\em one-off} edge in $S$ in a triangulation pair $(S,T)$ is an edge that is not common to $S$ and $T$ but that flips directly to an edge in $T$.  By Lemma 3a of \cite{stt}, for any one-off edge, there is a geodesic from $S$ to $T$ that begins with that edge flip.  A triangulation pair of size $n$ (with $n-1$ edges) with no common edges and no one-off edges must have distance at least $n$ as each edge flip can create at most one new common edge, and if there are no one-off edges present then it will take at least $n$ steps to transform $S$ to $T$.  With this we can show that conflicts can rise along a geodesic by any specified amount.

\begin{theorem} \label{thmbiggap}
For any positive $k$, there are examples of triangulation pairs $(S,T)$ of size $n \geq k+7$ where every geodesic $\gamma$ from $S$ to $T$ with $\gamma = \{S=S_0, S_1, S_2, \ldots, S_k=T \}$  has the property that $S_1$ has $k$ more conflicts with $T$ than $S$ has with $T$.  \end{theorem}


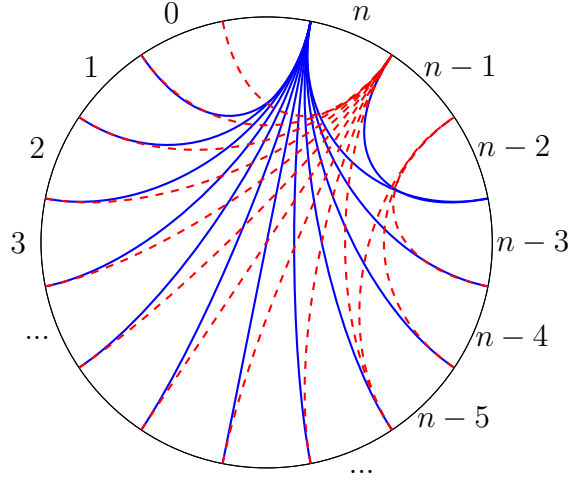
\begin{figure}
\begin{tikzpicture}[scale=3]
    \tikzset{chord/.style={blue, thick}}

    \draw (0,0) circle [radius=1];
	\coordinate(v0) at (-0.1951,0.9808) {};
	\coordinate(v1) at (-0.5556,0.8315) {};
	\coordinate(v2) at (-0.8315,0.5556) {};
	\coordinate(v3) at (-0.9808,0.1951) {};
	\coordinate(v4) at (-0.9808,-0.1951) {};
	\coordinate(v5) at (-0.8315,-0.5556) {};
	\coordinate(v6) at (-0.5556,-0.8315) {};
	\coordinate(v7) at (-0.1951,-0.9808) {};
	\coordinate(v8) at (0.1951,-0.9808) {};
	\coordinate(v9) at (0.5556,-0.8315) {};
	\coordinate(v10) at (0.8315,-0.5556) {};
	\coordinate(v11) at (0.9808,-0.1951) {};
	\coordinate(v12) at (0.9808,0.1951) {};
	\coordinate(v13) at (0.8315,0.5556) {};
	\coordinate(v14) at (0.5556,0.8315) {};
	\coordinate(v15) at (0.1951,0.9808) {};

	\draw[chord] (v1) .. controls (-0.2778,0.4157) and (0.0975,0.4904) .. (v15);
	\draw[chord] (v2) .. controls (-0.4157,0.2778) and (0.0975,0.4904) .. (v15);
	\draw[chord] (v3) .. controls (-0.4904,0.0975) and (0.0975,0.4904) .. (v15);
	\draw[chord] (v4) .. controls (-0.4904,-0.0975) and (0.0975,0.4904) .. (v15);
	\draw[chord] (v5) .. controls (-0.4157,-0.2778) and (0.0975,0.4904) .. (v15);
	\draw[chord] (v6) .. controls (-0.2778,-0.4157) and (0.0975,0.4904) .. (v15);
	\draw[chord] (v7) .. controls (-0.0975,-0.4904) and (0.0975,0.4904) .. (v15);
	\draw[chord] (v8) .. controls (0.0975,-0.4904) and (0.0975,0.4904) .. (v15);
	\draw[chord] (v9) .. controls (0.2778,-0.4157) and (0.0975,0.4904) .. (v15);
	\draw[chord] (v10) .. controls (0.4157,-0.2778) and (0.0975,0.4904) .. (v15);
	\draw[chord] (v11) .. controls (0.4904,-0.0975) and (0.0975,0.4904) .. (v15);
	\draw[chord] (v12) .. controls (0.4904,0.0975) and (0.0975,0.4904) .. (v15);
	\draw[chord] (v12) .. controls (0.4904,0.0975) and (0.2778,0.4157) .. (v14);
	
	 \tikzset{chord/.style={red, dashed, thick}}

    \draw (0,0) circle [radius=1];
	\node(0) at (-0.4210,1.0163) {$0$};
	\node(1) at (-0.7778,0.7778) {$1$};
	\node(2) at (-1.0163,0.4210) {$2$};
	\node(3) at (-1.1000,-0.0000) {$3$};
	\node(4) at (-1.0163,-0.4210) {$...$};
	\node(8) at (0.4210,-1.0163) {$...$};
	\node(9) at (0.8278,-0.7778) {$n-5$};
	\node(10) at (1.0863,-0.4210) {$n-4$};
	\node(11) at (1.1800,0.0000) {$n-3$};
	\node(12) at (1.0963,0.4210) {$n-2$};
	\node(13) at (0.8578,0.7778) {$n-1$};
	\node(14) at (0.4210,1.0163) {$n$};
	\coordinate(v0) at (-0.1951,0.9808) {};
	\coordinate(v1) at (-0.5556,0.8315) {};
	\coordinate(v2) at (-0.8315,0.5556) {};
	\coordinate(v3) at (-0.9808,0.1951) {};
	\coordinate(v4) at (-0.9808,-0.1951) {};
	\coordinate(v5) at (-0.8315,-0.5556) {};
	\coordinate(v6) at (-0.5556,-0.8315) {};
	\coordinate(v7) at (-0.1951,-0.9808) {};
	\coordinate(v8) at (0.1951,-0.9808) {};
	\coordinate(v9) at (0.5556,-0.8315) {};
	\coordinate(v10) at (0.8315,-0.5556) {};
	\coordinate(v11) at (0.9808,-0.1951) {};
	\coordinate(v12) at (0.9808,0.1951) {};
	\coordinate(v13) at (0.8315,0.5556) {};
	\coordinate(v14) at (0.5556,0.8315) {};
	\coordinate(v15) at (0.1951,0.9808) {};

	\draw[chord] (v0) .. controls (-0.0975,0.4904) and (0.2778,0.4157) .. (v14);
	\draw[chord] (v1) .. controls (-0.2778,0.4157) and (0.2778,0.4157) .. (v14);
	\draw[chord] (v2) .. controls (-0.4157,0.2778) and (0.2778,0.4157) .. (v14);
	\draw[chord] (v3) .. controls (-0.4904,0.0975) and (0.2778,0.4157) .. (v14);
	\draw[chord] (v4) .. controls (-0.4904,-0.0975) and (0.2778,0.4157) .. (v14);
	\draw[chord] (v5) .. controls (-0.4157,-0.2778) and (0.2778,0.4157) .. (v14);
	\draw[chord] (v6) .. controls (-0.2778,-0.4157) and (0.2778,0.4157) .. (v14);
	\draw[chord] (v7) .. controls (-0.0975,-0.4904) and (0.2778,0.4157) .. (v14);
	\draw[chord] (v8) .. controls (0.0975,-0.4904) and (0.2778,0.4157) .. (v14);
	\draw[chord] (v9) .. controls (0.2778,-0.4157) and (0.2778,0.4157) .. (v14);
	\draw[chord] (v9) .. controls (0.2778,-0.4157) and (0.4157,0.2778) .. (v13);
	\draw[chord] (v10) .. controls (0.4157,-0.2778) and (0.4157,0.2778) .. (v13);
	\draw[chord] (v11) .. controls (0.4904,-0.0975) and (0.4157,0.2778) .. (v13);

\end{tikzpicture}
\caption{Triangulation pairs with more conflicts arising along the mandatory first flip along a geodesic from $S=(10)^{i+3}11000$ in solid blue to $T=1(10)^{i}1101010000$ in dashed red. \label{bigexamples}}
\end{figure}

\begin{proof}
We consider triangulation pairs that are larger versions of the $S$ and $T$ used in Section \ref{secexample}. For a given $k$, we let $S$ be the triangulation whose dual tree has encoding $(10)^{k-2}11000$ and $T$ be that corresponding to $1(10)^{k-5}1101010000$.  These triangulations are shown in Figure \ref{bigexamples}.  The triangulation $S$ has edges from the counterclockwise end of interval $n$ to vertices at the ends of intervals $0$ through $n-3$, and one additional edge from the end of interval $n-3$ to $n-1$.  The triangulation $T$ has edges from the end of interval $n-1$ to the vertices at the end of the root interval through the end of interval $n-6$, and three edges from vertex $n-2$ to $n-6$ through $n-4$. There are no common edges or one-off edges in $S$ with respect to $T$, thus the distance between them is at least $n$.  Flipping the edge in $S$ from  $n-3$ to $n-1$ creates a comb from the end of $n$ and turns the edge from $n$ to $n-3$ in $S$ into a one-off edge.  Flipping that edge results in another one-off edge, and in fact a sequence of one-off edges is created by successively creating the common edges, giving a path $\gamma$ of length $n$ from $S$ to $T$.  Since we know the distance is at least $n$, the path $\gamma$ is a geodesic and the distance is indeed $n$.  
For other possible initial flips along alternate paths $\gamma'$, we consider first
the edge flip in $S$ of the edge from $n$ to $n-3$ to an edge from $n-1$ to $n-4$ to get triangulation $S'$.  This results in no new common edge, but a new one-off edge.  Following the one-off move and flipping the edge from $n-3$ to $n-1$ to an edge from $n-4$ to $n-2$ then results in a triangulation $S''$ that has a common edge with $T$ but $n-2$ edges where none is a one-off edge, meaning that the distance from from $S''$ to $T$ is at least $n-1$, resulting in a total length of at least $n+1$ from $S$ to $T$  and  thus longer than that for $\gamma$. Thus  $\gamma'$ or any path that begins with that same initial edge flip cannot be a geodesic.

 Any other initial edge-flip in $S$ does not create a one-off edge, and would result in a triangulation pair $(S',T)$ with no common or one-off edges and necessarily be still at distance at least $n$, resulting in a path of length $n+1$ or longer.  So any geodesic from $S$ to $T$ must begin with the same edge flip as $\gamma$ (in fact, this geodesic $\gamma$ is unique as can be seen by analyzing the following steps along the geodesic for possible created common edges.)  

The number of conflicts increases by $k$ for that first edge flip as the edge from $n-3$ to $n-1$ crosses 3 edges and it is flipped to the edge from $n-2$ to $n$ that conflicts with $n-4$ edges, giving the required rise of $k$ conflicts along every possible geodesic from $S$ to $T$ for examples of size $k+7$.  For larger examples, again we concatenate common identical triangles as needed to reach the desired size.
\end{proof}

Asymptotically, the fact that the number of conflicts can rise by $n-7$ is remarkable in light of the fact that the greatest number of conflicts a single edge can have is $n-1$ if it crosses all other edges in the other triangulation.

Again, we can concatenate two triangulations to create larger more symmetric examples where the number of conflicts must increase in either direction by any specified amount.

\begin{theorem}
For any positive $k$ and any  $n \geq 2k+14$ , there are triangulation pairs $(S,T)$  of size $n$ where every geodesic $\gamma$ from $S$ to $T$  or from $T$ to $S$ has a step along the geodesic where the number of conflicts increases by $k$.
\end{theorem}

Thus, there are large pairs where the number of conflicts along geodesics must rise in a single step an amount proportional to the size for geodesics in either direction.

\begin{table}
{\tiny
\begin{tabular}{|l|l|| l | l |}
\hline 
\hspace{.45in} S &\hspace{.45in} T & \hspace{.45in} S &\hspace{.45in} T \\
\hline
 10101010101011000 & 11010101101010000 * &
 10101010101100100 & 11010101101001000  \\
 10101010101100100 & 11010111010100000  &
 10101010101100100 & 11011011010100000  \\
 10101010101100100 & 11101011010100000  &
 10101010110010100 & 11010101100110000  \\
 10101010110010100 & 11010111010010000  &
 10101010110010100 & 11011011010010000  \\
 10101010110010100 & 11011110101000000  &
 10101010110010100 & 11101011010010000  \\
 10101010110010100 & 11101110101000000  &
 10101010110010100 & 11110110101000000  \\
 10101010110011000 & 11011110101000000 * &
 10101010110011000 & 11110110101000000  \\
 10101011001010100 & 11010101011100000  &
 10101011001010100 & 11010111001100000  \\
 10101011001010100 & 11011011001100000  &
 10101011001010100 & 11011110100100000  \\
 10101011001011000 & 11011110100100000 *  &
 10101011001011000 & 11110110100100000  \\
 10101011001011000 & 11111101010000000 * &
 10101011001100100 & 11111101010000000 * \\
 10101011001101000 & 11111101010000000  &
 10101011001110000 & 11111101010000000  \\
 10101011110001000 & 11010100110100100  &
 10101011110010000 & 11010100110011000  \\
 10101011110100000 & 11010100101110000  &
 10101100101011000 & 11011110011000000  \\
 10101100101011000 & 11110110011000000 * &
 10101100101011000 & 11111101001000000  \\
 10101100101100100 & 11111101001000000  &
 10101100101101000 & 11111101001000000  \\
 10101101010110000 & 11010110101000100  &
 10101101011001000 & 11010110100100100  \\
 10101111100100000 & 11010010101110000  &
 10101111100100000 & 11010011100110000  \\
 10101111100100000 & 11011000101110000 * &
 10101111100100000 & 11110111100000000  \\
 10101111101000000 & 11010011011100000 *  &
 10110010101100100 & 11010101101001000 * \\
 
 \hline
\end{tabular}}
\caption{Representatives of the 40 tree pair equivalence classes of size 8 where all neighbors of $S$ which are closer to $T$ in the rotation graph have more conflicts with $T$ than $S$ does. Those pairs marked with asterisks are drawn as triangulation pairs in Figure \ref{fig:selectedpairs}. \label{bad8}}
\end{table}
We do note that this kind of behavior where every geodesic has an increase in conflicts along the path somewhere are somewhat rare.  In size 8, there are 40 equivalence classes  of examples where the first step along any geodesic from $S$ to $T$ will necessarily increase the number of conflicts, from among the 117,260 equivalence classes of edge-flip distance problems of size 8.  Representatives of each of these equivalence classes are tabulated in Table \ref{bad8}, with a selection of those pairs drawn in Figure \ref{fig:selectedpairs}.
For size 9, the corresponding fraction is also small- there are 632 equivalence classes where all geodesics have an increase in conflicts at the first step, out of 1,108,536 classes of problems.

\begin{figure}
    \centering
    \begin{tabular}{ccc}
        \subfloat[
                 ]{\includegraphics[scale=\scaleSizeB]{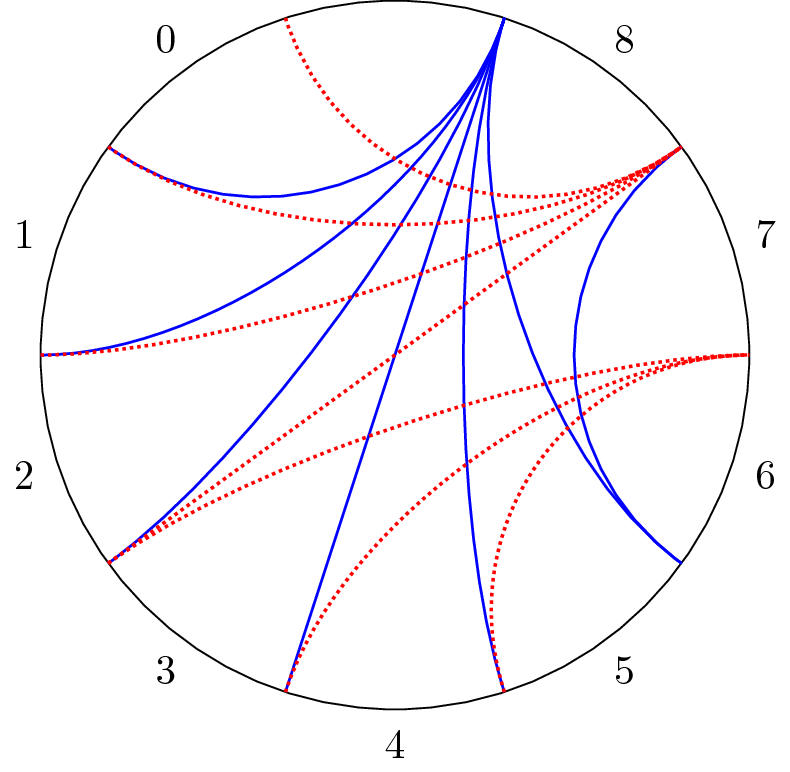}} &
        \subfloat[
                 ]{\includegraphics[scale=\scaleSizeB]{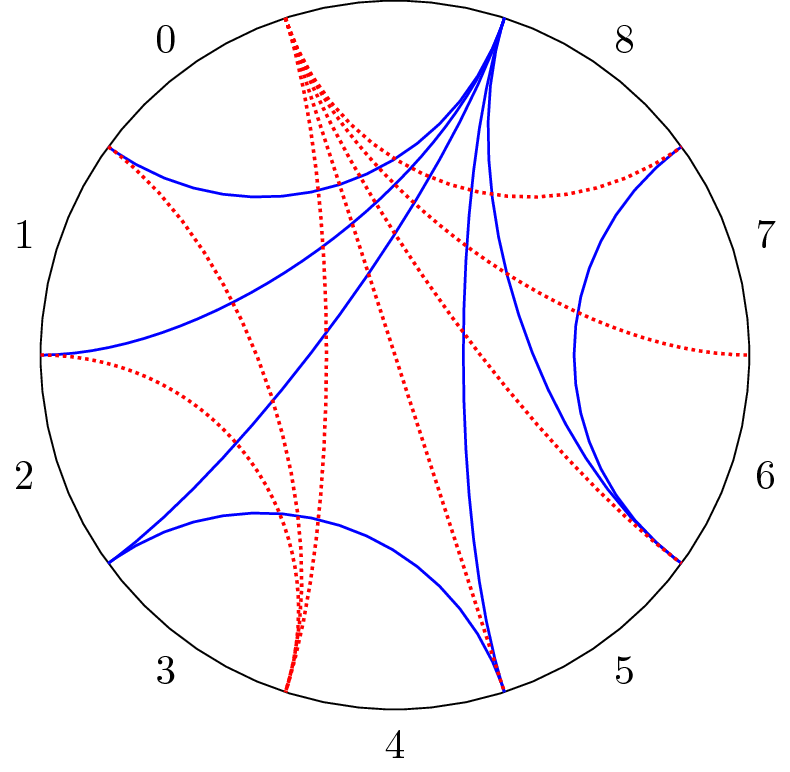}} &
        \subfloat[
                 ]{\includegraphics[scale=\scaleSizeB]{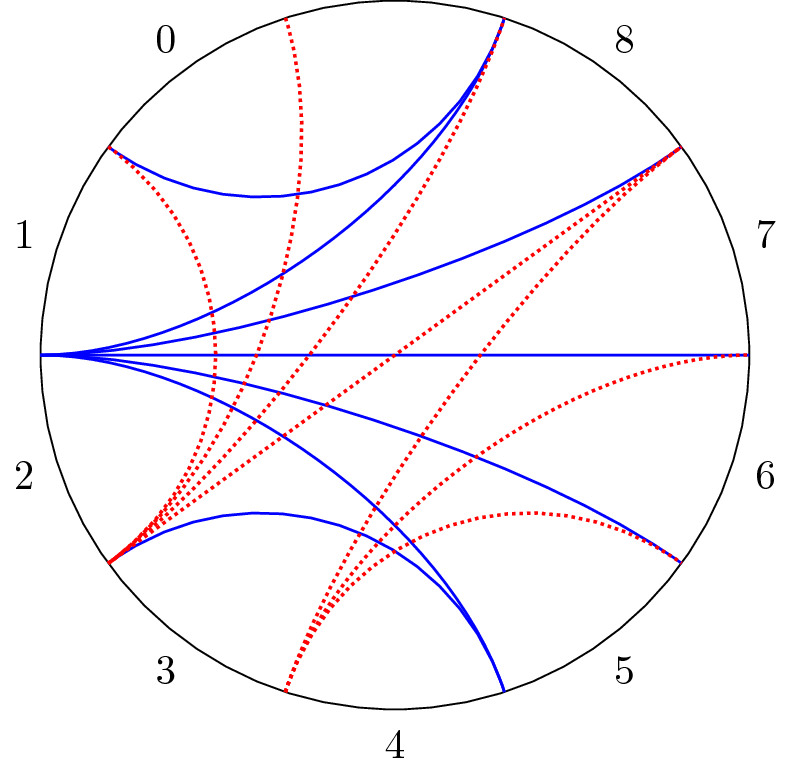}} \\
        \subfloat[
                 ]{\includegraphics[scale=\scaleSizeB]{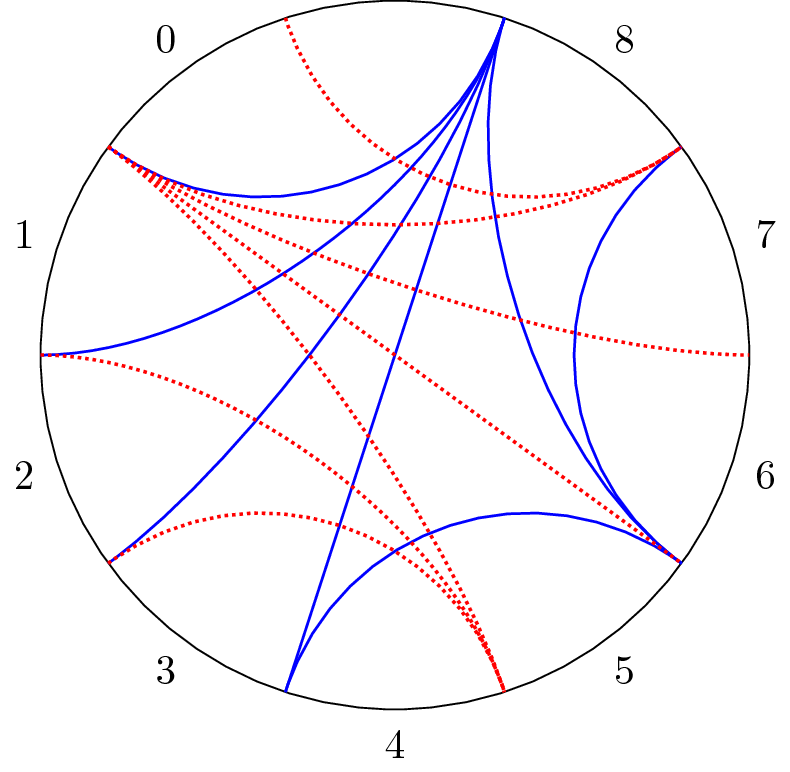}} &
        \subfloat[
                 ]{\includegraphics[scale=\scaleSizeB]{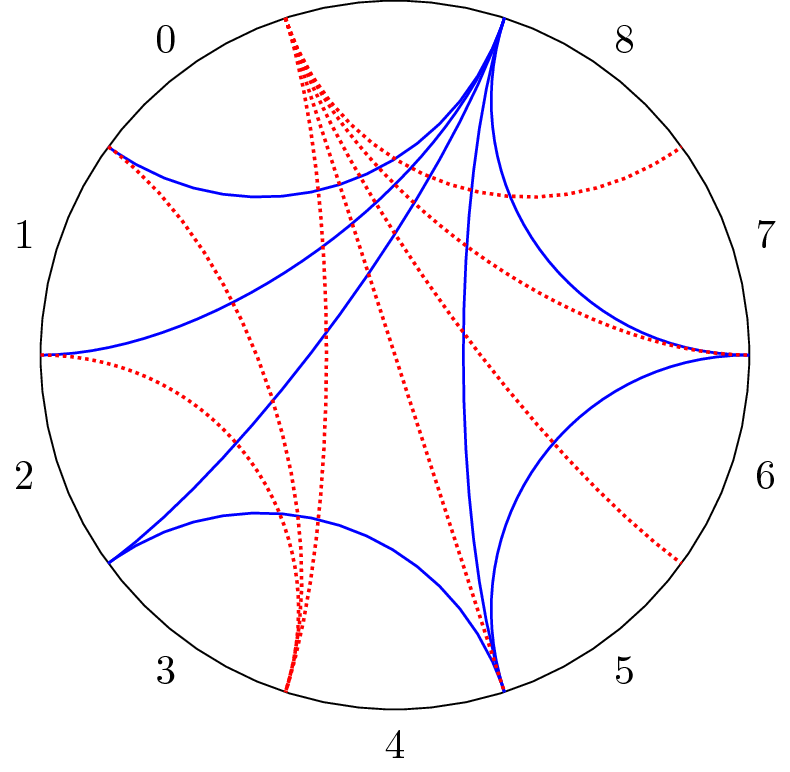}} &
        \subfloat[
                 ]{\includegraphics[scale=\scaleSizeB]{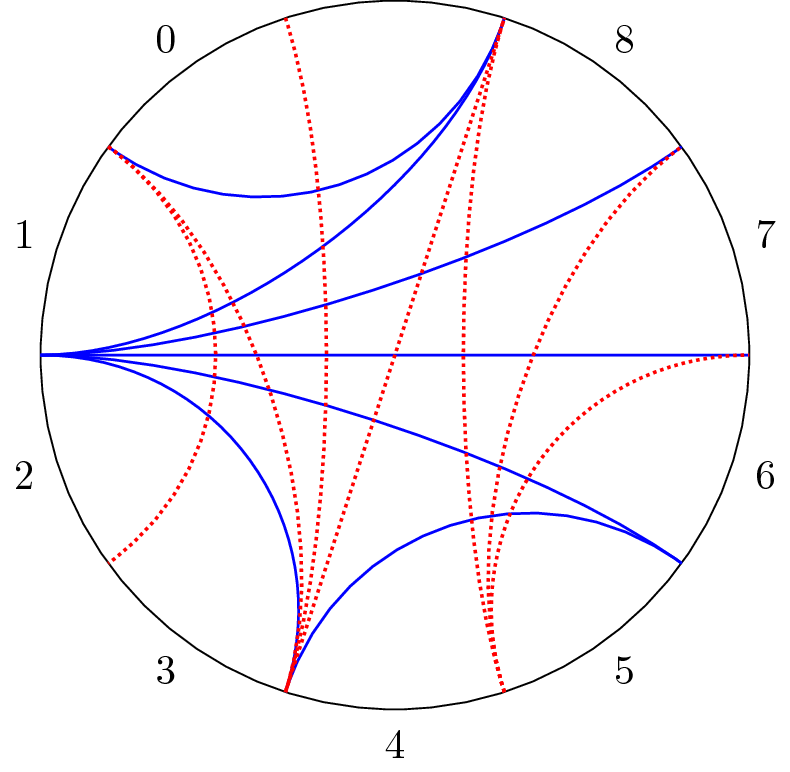}} \\
        \subfloat[
                 ]{\includegraphics[scale=\scaleSizeB]{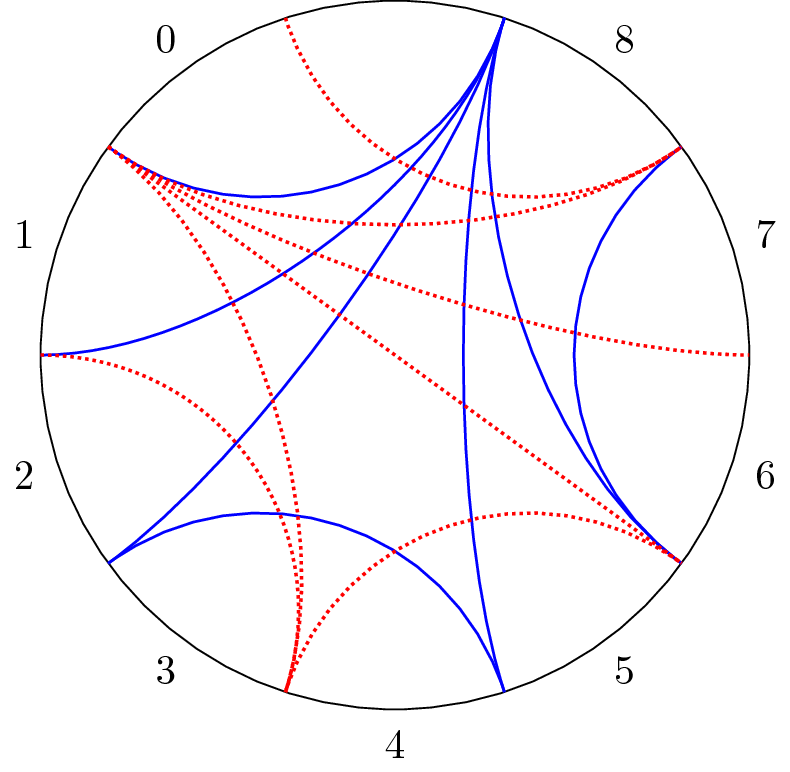}} &
        \subfloat[
                 ]{\includegraphics[scale=\scaleSizeB]{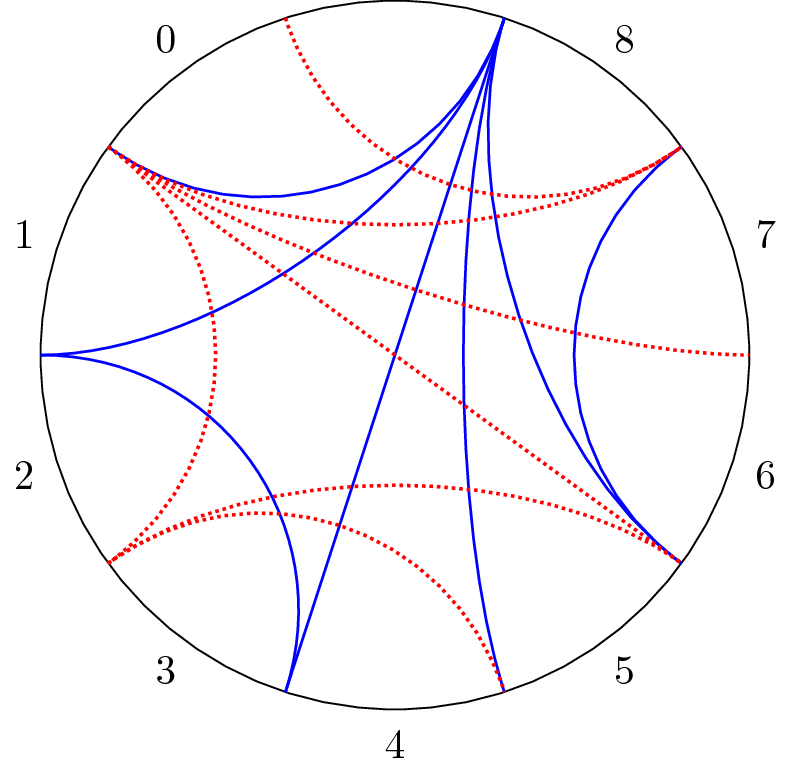}} &
        \subfloat[
                 ]{\includegraphics[scale=\scaleSizeB]{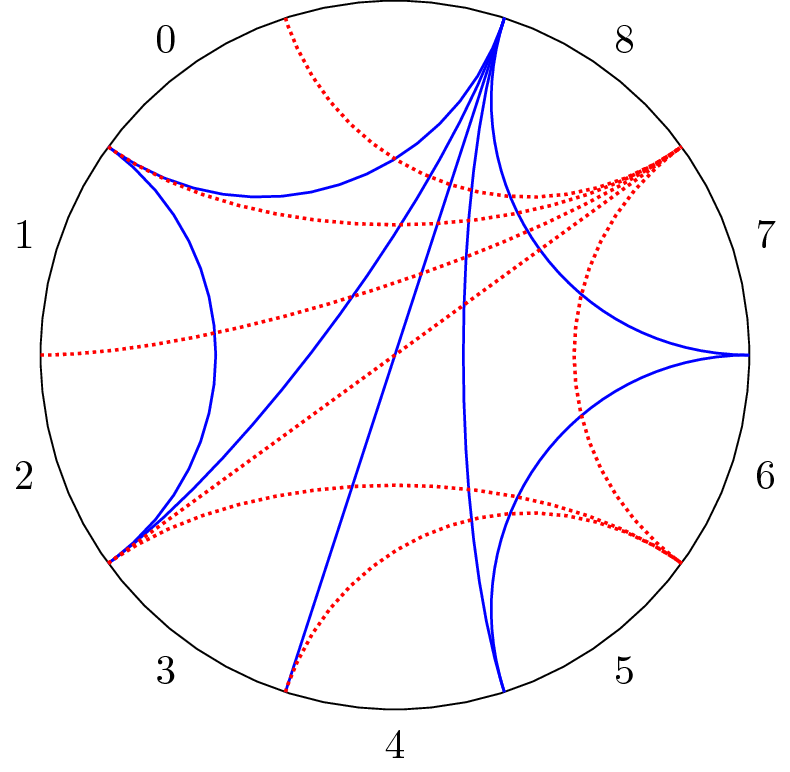}}
    \end{tabular}
    \caption{The asterisked tree pairs from Table \ref{bad8}. These selected pairs illustrate a number of differences, including the number of possible geodesics varying from $1$ to $6$. The pairs in the first row have a unique geodesic, those in the second row have exactly $2$ geodesics, and in the third row all have $6$. The chords rendered in solid blue are those of $S$ and the chords rendered in dotted red are those of $T$. Each geodesic proceeds by flipping the chords of $S$.}
    \label{fig:selectedpairs}
\end{figure}

Though the examples above show that conflict-reducing algorithms do not always give the correct distances, they often do give correct distances and when they do not, typically the gap between  the simplest conflict-based estimate and the true geodesic distance is generally not large.  

In the case of the 117,260 equivalence classes of size 8 edge-flip distance problems of which the earlier example $(S,T)$ is one instance, the lexicographically-least minimal conflict greedy algorithm is correct in 111,061 cases, with 5771 cases overestimating the distance by 1, 423 cases of an overestimate by 2, and the furthest it is ever off by is 3, which happens in 5 cases.  That results in about a 6\% overestimate of distance on average across all problems of size 8.

For the 1,108,536 equivalence classes of distance problems of size 9, despite the 632 cases where  all the geodesics begin with an increase in conflicts, the naive greedy algorithm gives the geodesic distance correctly in more than 1 million cases.  In the cases where it is incorrect, in 4 cases  the overestimate of the distance is 4, in 444 cases the overestimate is  3, in 7047 cases it is off by 2, and  in  80,710 cases the algorithm overestimates by 1.   On average, this greedy algorithm overstates the distance by about .0867378.  Similarly, in the case of size 10 for triangulations of the 11-gon,  the vast majority of distances (89\%) are correctly determined with an average of 0.11888 overestimate of the distance.

%

\bibliographystyle{plain}

\begin{thebibliography}{10}

\bibitem{barilpallo}
Jean-Luc Baril and Jean-Marcel Pallo.
\newblock Efficient lower and upper bounds of the diagonal-flip distance
  between triangulations.
\newblock {\em Information Processing Letters}, 100(4):131--136, 2006.

\bibitem{conflicts}
Timothy Chu and Sean Cleary.
\newblock Expected conflicts in pairs of rooted binary trees.
\newblock {\em Involve}, 6(3):323--332, 2013.

\bibitem{commonedges}
Sean Cleary, Andrew Rechnitzer, and Thomas Wong.
\newblock Common edges in rooted trees and polygonal triangulations.
\newblock {\em Electron. J. Combin.}, 20(1):Paper 39, 22, 2013.

\bibitem{rotfpt}
Sean Cleary and Katherine St.~John.
\newblock Rotation distance is fixed-parameter tractable.
\newblock {\em Inform. Process. Lett.}, 109(16):918--922, 2009.

\bibitem{clearylinear}
Sean Cleary and Katherine St.~John.
\newblock A linear-time approximation for rotation distance.
\newblock {\em J. Graph Algorithms Appl.}, 14(2):385--390, 2010.

\bibitem{cw}
Karel Culik and Derick Wood.
\newblock A note on some tree similarity measures.
\newblock {\em Information Processing Letters}, 15(1):39--42, 1982.

\bibitem{hankeottomannschuierer}
Sabine Hanke, Thomas Ottmann, and Sven Schuierer.
\newblock The edge-flipping distance of triangulations.
\newblock {\em J.UCS}, 2(8):570--579 (electronic), 1996.

\bibitem{pournin}
Lionel Pournin.
\newblock The diameter of associahedra.
\newblock {\em Adv. Math.}, 259:13--42, 2014.

\bibitem{charlesmike}
Charles Semple and Mike Steel.
\newblock {\em Phylogenetics}, volume~24 of {\em Oxford Lecture Series in
  Mathematics and its Applications}.
\newblock Oxford University Press, Oxford, 2003.

\bibitem{stt}
Daniel~D. Sleator, Robert~E. Tarjan, and William~P. Thurston.
\newblock Rotation distance, triangulations, and hyperbolic geometry.
\newblock {\em J. Amer. Math. Soc.}, 1(3):647--681, 1988.

\bibitem{stanley1}
Richard~P. Stanley.
\newblock {\em Enumerative combinatorics. {V}ol. {I}}.
\newblock The Wadsworth \& Brooks/Cole Mathematics Series. Wadsworth \&
  Brooks/Cole Advanced Books \& Software, Monterey, CA, 1986.
\newblock With a foreword by Gian-Carlo Rota.

\end{thebibliography}

\end{document}